\documentclass[11pt,a4paper]{article}
\usepackage[top=2.5cm,bottom=2.5cm,left=2.5cm,right=2.5cm]{geometry}

\usepackage{amssymb}
\usepackage{amssymb,amsmath,graphicx,color,amsthm}

\newtheorem{theorem}{Theorem}
\newtheorem{lemma}[theorem]{Lemma}
\newtheorem{corollary}[theorem]{Corollary}

\newtheorem{conjecture}{Conjecture}
\newtheorem{question}{Question}

\newcommand{\set}[1]{\ensuremath{\left\{#1 \right\}}}

\newcommand{\ci}{\ensuremath{\chi_{\mathrm{irr}}'}}

\usepackage{url}
\makeatletter
\g@addto@macro{\UrlBreaks}{\UrlOrds}
\makeatother

\newenvironment{proofclaim}[1][]%
    {\noindent \emph{Proof.} {}{#1}{}}{$~$\hfill $~\blacklozenge$ \vspace{0.2cm}}

\begin{document}

\title{New bounds for locally irregular chromatic index of bipartite and subcubic graphs}

\author
{
	Borut Lu\v{z}ar\thanks{Faculty of Information Studies, Novo mesto, Ljubljanska cesta 31A, 8000 Novo mesto, Slovenia. 
		E-Mail: \texttt{borut.luzar@gmail.com}}, \
	Jakub Przyby{\l}o\thanks{AGH University of Science and Technology, al. A. Mickiewicza 30, 30--059 Krakow, Poland. 
		E-Mail: \texttt{jakubprz@agh.edu.pl}}, \
	Roman Sot\'{a}k\thanks{Faculty of Science, Pavol Jozef \v{S}af\'arik University, Jesenn\'{a} 566/5, 040 01 Ko\v{s}ice, Slovakia.
		E-Mail: \texttt{roman.sotak@upjs.sk}}
}

\maketitle

{
\begin{abstract}
	A graph is \textit{locally irregular} if the neighbors of every vertex $v$ have degrees distinct from the degree of $v$.
	A \textit{locally irregular edge-coloring} of a graph $G$ is an (improper) edge-coloring such that
	the graph induced on the edges of any color class is locally irregular. 
	It is conjectured that $3$ colors suffice for a locally irregular edge-coloring. 
	Recently, Bensmail et al. 
	(Bensmail, Merker, Thomassen: Decomposing graphs into a constant number of locally irregular subgraphs, {\em European J. Combin.}, 60:124--134, 2017)
	settled the first constant upper bound for the problem to $328$ colors.
	In this paper, using a combination of existing results, 
	we present an improvement of the bounds for bipartite graphs and general graphs, setting the best upper bounds
	to 7 and 220, respectively. 
	In addition, we also prove that $4$ colors suffice for locally irregular edge-coloring of any subcubic graph. 
\end{abstract}
}


\section{Introduction}

In this paper, we consider only finite simple graphs.
A graph is \textit{locally irregular} if the neighbors of every vertex $v$ have degrees distinct from the degree of $v$.
A \textit{locally irregular edge-coloring} of a graph $G$ is an (improper) edge-coloring 
such that the graph induced on the edges of any color class is locally irregular. 
The coloring has been recently introduced by Baudon et al.~\cite{BauBenPrzWoz15} 
who were motivated by the well known (1-2-3)-conjecture proposed in~\cite{KarLucTho04}.

Clearly, not every graph admits a locally irregular edge-coloring. Beside trivial examples (paths and cycles of odd lengths),
in~\cite{BauBenPrzWoz15}, the authors completely characterized the graphs not admitting any locally irregular edge-coloring. 
The graphs that do, are called \textit{decomposable}. The smallest number of colors such that a decomposable graph $G$ admits 
a locally irregular edge-coloring is called the \textit{locally irregular chromatic index} and denoted by $\ci(G)$.
The initiators proposed the following surprising conjecture.
\begin{conjecture}[Baudon et al., 2015]
	\label{con:main}
	Let $G$ be a decomposable graph. Then
	$$
		\ci(G) \le 3\,.
	$$
\end{conjecture}

Recently, Bensmail et al.~\cite{BenMerTho16} established the first constant upper bound for general graphs.
\begin{theorem}[Bensmail et al., 2016]
	Let $G$ be a decomposable graph. Then
	$$
		\ci(G) \le 328\,.
	$$	
\end{theorem}
On the other hand, Conjecture~\ref{con:main} has already been confirmed for several special classes of graphs, 
e.g. graphs with minimum degree at least $10^{10}$~\cite{Prz15}
and $k$-regular graphs, for $k \ge 10^{7}$~\cite{BauBenPrzWoz15}.

In this paper we consider two particular classes of graphs, bipartite and subcubic, establishing some new upper bounds 
and confirming Conjecture~\ref{con:main} for several subclasses. 

The locally irregular edge-coloring of bipartite graphs has already been intensively studied. 
Conjecture~\ref{con:main} has been proven in affirmative for decomposable trees~\cite{BauBenPrzWoz15}, 
and just recently, 
Baudon et al.~\cite{BauBenSop15} introduced a linear-time algorithm for determining the irregular chromatic index of any tree.
Note that there exist trees with locally irregular chromatic index equal to $3$ (see Fig.~\ref{fig:tree}).
\begin{figure}[htp!]
	$$
		\includegraphics{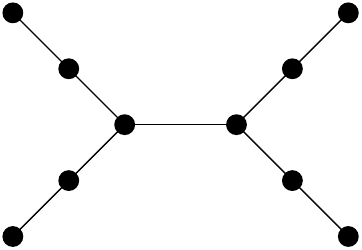}
	$$
	\caption{A tree with locally irregular chromatic index equal to $3$.}
	\label{fig:tree}
\end{figure}
However, if the maximum degree of a tree is at least $5$, it admits a locally irregular edge-coloring with at most two colors~\cite{BauBenSop15}.

In~\cite{BauBenPrzWoz15}, using a result from~\cite{HavParSam14}, the authors observed the following result 
regarding regular bipartite graphs.
\begin{theorem}[Baudon et al., 2015]
	Let $G$ be a regular bipartite graph with minimum degree at least $3$. Then
	$$
		\ci(G) \le 2\,.
	$$			
\end{theorem}
Bensmail et al.~\cite{BenMerTho16} even asked if two colors suffice for all such graphs with the regularity assumption removed.
Moreover, they established an upper bound for general bipartite graphs.
\begin{theorem}[Bensmail et al., 2016]
	\label{thm:bipold}
	Let $G$ be a decomposable bipartite graph. Then
	$$
		\ci(G) \le 10\,.
	$$				
\end{theorem}
They proved the bound above by showing that bipartite graphs with an even number of edges admit a locally irregular edge-coloring
with at most $9$ colors. 
The aim of Section~\ref{sec:bip} is to improve both above results as follows.
\begin{theorem}
	\label{thm:mainbip}
	Let $G$ be a decomposable bipartite graph. Then
	$$
		\ci(G) \le 7\,.
	$$				
	Moreover, if $G$ has an even number of edges, then the upper bound is $6$.
\end{theorem}

Consequently, the improvement in Theorem~\ref{thm:mainbip}, also gives the following bound for general graphs 
(again analogously to the proof of Theorem~$4.6$ in~\cite{BenMerTho16}).
\begin{theorem}
	\label{thm:gen}
	Let $G$ be a decomposable graph. Then
	$$
		\ci(G) \le 220\,.
	$$		
\end{theorem}

In the second part of the paper, in Section~\ref{sec:cubic}, we consider graphs with maximum degree $3$. 
We expect Conjecture~\ref{con:main} to be true for this class of graphs and we make the first step toward proving it by establishing
an upper bound of $4$.
\begin{theorem}
	\label{thm:main}
	Let $G$ be a decomposable subcubic graph. Then 
	$$
		\ci(G) \le 4\,.
	$$
\end{theorem}
We also believe that our method might be appropriate for further investigations of the locally irregular edge-coloring of graphs with small maximum degrees.


\section{Bipartite graphs}
	\label{sec:bip}

In order to prove Theorem~\ref{thm:mainbip}, we will use several auxiliary results.
Following the definition of Bensmail et al.~\cite{BenMerTho16}, 
we say that a decomposable bipartite graph is \textit{balanced} 
if all the vertices in one of the two partition parts have even degrees.
The upper bound for locally irregular chromatic index of balanced forests is given in the following lemma.
\begin{lemma}[Bensmail et al., 2016 (Lemma 3.2)]
	\label{lem:bal}
	Let $F$ be a balanced forest. Then $F$ admits a locally irregular edge-coloring with at most $2$ colors.
	Moreover, for each vertex $v$ in the partition with no vertex of odd degree, all edges incident to $v$ have the same color.
\end{lemma}

The following result is of a similar flavor. 
A \textit{fully subdivided graph} $\mathcal{S}(G)$ is a graph obtained from a graph $G$ by subdividing every edge in $G$ once.
Observe that $\mathcal{S}(G)$ is decomposable for every graph $G$.
\begin{theorem}
	\label{thm:sub}
	Let $G$ be a (multi)graph not isomorphic to an odd cycle. Then
	$$
		\ci(\mathcal{S}(G)) \le 2\,.
	$$
\end{theorem}

\begin{proof}	
	Let $G$ be a counter-example to the theorem with the minimum number of edges.	
	If $G$ is a tree, then the result follows from Lemma~\ref{lem:bal}.	
    So, we may assume that $\mathcal{S}(G)$ contains an even cycle $\mathcal{S}(C)$. 
    Now, consider two cases: $G-C$ contains no component isomorphic to an odd cycle or there is an odd cycle $C'$ as a component. 

	In the former case, by the minimality, 
	the graph $\mathcal{S}(G) - \mathcal{S}(C)$ admits a locally irregular edge-coloring $\varphi$ with at most $2$ colors. 
	Color the edges of $\mathcal{S}(C)$ in such a way that the two edges incident with any vertex $v$ of degree at least $3$ 
    are colored by an arbitrary color used at $v$ in $\varphi$. 
    Finally, we complete the coloring $\mathcal{S}(G)$ by coloring the rest of the edges,
    all incident with the vertices of degree $2$ in $G$. It is easy to see that this is always possible.
	
	In the latter case, we consider two subcases. If $G-C'$ is not an odd cycle, 
	we proceed as in the former case with the graph $\mathcal{S}(G) - \mathcal{S}(C')$. 
	Hence, we may assume that $G$ is a union of two odd cycles. 
	If $C$ and $C'$ have only one vertex in common, then there is a locally irregular edge-coloring of $\mathcal{S}(G)$
	such that we color the edges incident to the vertex of degree $4$ with one color, and complete the coloring of the two remaining 
	paths.
	If $C$ and $C'$ have more than one vertex in common, 
	then there is even cycle $C''$ in $G$ and we consider $C''$ as the cycle in induction.     
\end{proof}

We are also confident that the following question has an affirmative answer.
\begin{question}
	Is every connected bipartite graph with all vertices in one partition of even size, which is not a cycle of length $4k+2$, locally irregular $2$-edge-colorable?
\end{question}

We continue by introducing a type of edge-colorings presented in~\cite{LuzPetSkr16}.
Given a graph $G$, a mapping $\pi:V(G)\to \{0,1\}$ is a \textit{vertex signature} for $G$, and a pair $(G, \pi)$ is called a \textit{parity pair}.
A \textit{vertex-parity edge-coloring} of a parity pair $(G,\pi)$ is a (not necessarily proper) edge-coloring
such that at every vertex $v$ each appearing color $c$ is in parity accordance with $\pi$, i.e. the number of edges of color $c$ incident to $v$ is even if $\pi(v) = 0$,
and odd if $\pi(v) = 1$. 
The \textit{vertex-parity chromatic index} $\chi_{p}'(G,\pi)$ is the least integer $k$ for which $(G,\pi)$ is $k$-edge-colorable.
Clearly, not every parity pair admits a vertex-parity edge-coloring.
The following two are necessary conditions for the existence of $\chi_{p}'(G,\pi)$:
\begin{itemize}
	\item [$(P_1)$] Every vertex $v$ of $(G,\pi)$ with $\pi(v) = 0$ has even degree in $G$.
	\item [$(P_2)$] In every component of $G$, there are zero or at least two vertices with the vertex signature value $1$.
\end{itemize}
Whenever $(P_1)$ and $(P_2)$ are fulfilled, $\pi$ is a \textit{proper vertex signature} of $G$ and $(G,\pi)$ is a \textit{proper parity pair}.
In~\cite{LuzPetSkr16}, it was proven that $6$ colors always suffice for a vertex-parity edge-coloring of any (multi)graph, and a characterization of the graphs with 
the vertex-parity chromatic index equal to five and six was given. In particular, by simplifying Corollary~$2.7$ from~\cite{LuzPetSkr16}, we have the following.
\begin{theorem}[Lu\v{z}ar et al., 2016]
    \label{thm:parity}
    Let $G$ be a connected graph, and let $(G,\pi)$ be a proper parity pair. If $|\pi^{-1}(1)| \neq 3$, then 
    $$
    	\chi_{p}'(G,\pi) \leq 4\,.
    $$
\end{theorem}

This immediately implies the result for bipartite graphs with one partition containing only vertices of even degrees.
\begin{corollary}
	\label{cor:par}
	Let $G$ be a balanced graph. Then
	$$
		\ci(G) \le 4\,.
	$$	
\end{corollary}

\begin{proof}	
	Let $(\mathcal{E},\mathcal{N})$ be a bipartition of a balanced (without loss of generality connected) graph $G$, where the degree of every vertex in $\mathcal{E}$ is even.
	We consider two cases regarding the cardinality of $\mathcal{N}$. 
	If $|\mathcal{N}| \neq 3$, then, by Theorem~\ref{thm:parity}, 
	there exists a vertex-parity edge-coloring $\varphi$ of a proper parity pair $(G,\pi)$ with at most $4$ colors,
	where the proper vertex signature $\pi$ is obtained by assigning $0$ to all vertices from $\mathcal{E}$ and $1$
	to the vertices of $\mathcal{N}$. Observe that $\varphi$ is a locally irregular edge-coloring of $G$, 
	since for every color $c$, the graph $G_c$ induced by the edges of $c$ is locally irregular: for every edge $e$ its two 
	end-vertices have degrees of different parities in $G_c$.
	
	Consider now the case when $|\mathcal{N}| = 3$, with $\mathcal{N} = \set{v_1,v_2,v_3}$. 
	Clearly, as we consider simple graphs, there are only vertices of degree $2$ in $\mathcal{E}$.
	Such a graph is a fully subdivided graph of some multigraph and so, by Theorem~\ref{thm:sub}, the result follows.
\end{proof}

Using the bound of Corollary~\ref{cor:par}, we are able to prove Theorem~\ref{thm:mainbip} 
simply by following the proofs of Theorem~$3.9$ and Corollary~$3.10$ in~\cite{BenMerTho16}.
Similarly, we obtain the result in Theorem~\ref{thm:gen}. For the sake of completeness, 
we recall the proof of Theorem~$4.6$ in~\cite{BenMerTho16} and modify it with our result.
\begin{proof}[Proof of Theorem~\ref{thm:gen}.]
	By Theorem~$2.3$ from~\cite{BenMerTho16}, it suffices to show that $\ci(G) \le 219$ holds
	for an even sized connected graph $G$.
	By Lemma~$4.5$~from~\cite{BenMerTho16}, $G$ can be decomposed into two graphs $D$ and $H$
	such that $D$ is $(2 \cdot 10^{10}+2)$-degenerate, every connected component of $D$ is of even size,
	and the minimum degree of $H$ is at least $10^{10}$. By the result of Przyby{\l}o~\cite{Prz15}, 
	$\ci(H) \le 3$ and by Theorem~$4.3$ from~\cite{BenMerTho16} and Theorem~\ref{thm:mainbip} it holds
	$$
		\ci(D) \le 6 (\lceil \log_2 (2\cdot 10^{10} + 3) \rceil + 1) = 216.
	$$
	Hence, $\ci(G) \le \ci(H) + \ci(D) \le 3 + 216 = 219$.
\end{proof}

\section{Subcubic graphs}
	\label{sec:cubic}

In this section we consider graphs with maximum degree $3$ 
and prove a stronger version of Theorem~\ref{thm:main}. First, we introduce some additional notation and auxiliary results.

Let $K_{1,3}''$ denote the complete bipartite graph $K_{1,3}$ with two edges subdivided once. 
An edge-decomposition of a connected graph is called \textit{pertinent} if it is comprised of paths of length $2$ ($2$-paths) 
and at most one element isomorphic either to $K_{1,3}$ or $K_{1,3}''$. If a graph is not connected, then its edge-decomposition 
is \textit{pertinent} if the restriction to every component of the graph is pertinent.
We will also use a structural result presented in~\cite{BenMerTho16} (Lemma~$2.2$ and Theorem~$2.3$).
\begin{theorem}[Bensmail et al., 2016]
	\label{thm:decompose}
	Every decomposable graph admits a pertinent edge-decomposition. 
\end{theorem}

In order to prove Theorem~\ref{thm:main}, we give a proof of a somewhat stronger statement.
We say that an edge-decomposition $\mathcal{D}$ of a graph is \textit{strongly pertinent} if it is pertinent and in the case 
$\mathcal{D}$ contains an element isomorphic to $K_{1,3}''$ in some component $C$, the graph has no pertinent edge-decomposition without $K_{1,3}''$ in $C$.
Two elements of an edge-decomposition are \textit{incident} if they have a common vertex.
\begin{theorem}
	\label{thm:mainext}
	Let $G$ be a decomposable subcubic graph and let $\mathcal{D}$ be a strongly pertinent edge-decomposition of $G$. 
	Then, $G$ admits a locally irregular edge-coloring with at most $4$ colors such that
	\begin{itemize}
		\item[$(i)$] the edges of every element of $\mathcal{D}$ are colored with the same color; and
		\item[$(ii)$] if the edges of two incident elements $p_1$, $p_2$ of $\mathcal{D}$ are colored with the same color, 
			then the vertex, at which $p_1$ and $p_2$ are incident, is the central vertex of either $p_1$ or $p_2$.
	\end{itemize}
\end{theorem}

\begin{proof}
	We proceed by contradiction. Let $G$ be a subcubic graph with the smallest number of edges such that it does
	not admit a locally irregular edge-coloring with at most $4$ colors satisfying the properties $(i)$ and $(ii)$. 
	Clearly, $G$ is connected.
	Let $\mathcal{D}$ be a strongly pertinent edge-decomposition of $G$.
	Given an element $q \in \mathcal{D}$, we say that it has $k$ \textit{conflicts} if it is incident to $k$ distinct elements of $\mathcal{D}$.
	Moreover, a vertex of degree $1$ in an element of $\mathcal{D}$ is called \textit{pendant}, the other vertices are called \textit{central} (see Fig.~\ref{fig:decom}).
	Note that $K_{1,3}''$ has three central vertices.
	\begin{figure}[htp!]
		$$
			\includegraphics{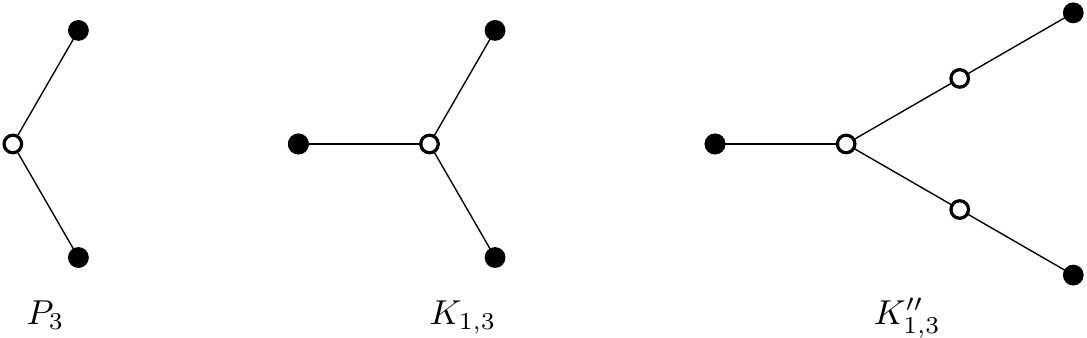}
		$$
		\caption{A $2$-path and the graphs $K_{1,3}$ and $K_{1,3}''$. The central vertices are depicted as empty circles; the pendant are depicted as full circles.}
		\label{fig:decom}
	\end{figure}	
	
	By the definition of strongly pertinent decompositions, we immediately infer the following.
	
	\medskip
	\noindent\textbf{Claim $1$.} \textit{No pendant vertex of $K_{1,3}''$ is the central vertex of any $2$-path in $\mathcal{D}$.}
	
	\smallskip
	\begin{proofclaim}
		Indeed, suppose there is a $2$-path $p$ in $\mathcal{D}$ incident at the central vertex to an element $q$ isomorphic to $K_{1,3}''$. Then, 
		there exists a pertinent decomposition with all the elements equal to the elements of $\mathcal{D}$
		except for $p$ and $q$. These two can be decomposed into a $K_{1,3}$ comprised of $p$ and the edge of $q$ incident to $p$, 
		and two $2$-paths comprised of the remaining four edges of $q$.
	\end{proofclaim}
	
	\medskip
	Clearly, if $G$ has at most four edges, then it admits a locally irregular edge-coloring with at most two colors satisfying the properties $(i)$ and $(ii)$.
	Hence, we may assume that $G$ has at least five edges.
	For simplicity, we abuse the notation by denoting the colors of the edges of an element $q$ from $\mathcal{D}$ by $\varphi(q)$.
	We say that a color $c$ is \textit{free} for an element if no incident element is colored with $c$.
	
	\medskip
	\noindent\textbf{Claim $2$.} \textit{Every element of $\mathcal{D}$ has at least $4$ conflicts.}
	
	\smallskip
	\begin{proofclaim}
		Suppose that there is an element $p \in \mathcal{D}$ with at most three conflicts. 
		Then, $G - p$ admits a locally irregular $4$-coloring $\varphi$, 
		and we complete the coloring of $G$ by coloring the edges of $p$ with the color $p$ is not incident with. 
	\end{proofclaim}
		
	\medskip
	Claim $2$ implies a simple observation regarding the number of central vertices incident to elements of $\mathcal{D}$.
	
	\medskip
	\noindent\textbf{Claim $3$.} \textit{At most one pendant vertex of a $2$-path is the central vertex of some element of $\mathcal{D}$, 
	and at most two pendant vertices of $K_{1,3}$ are the central vertices of $2$-paths in $\mathcal{D}$.}
	
	\medskip
	Hence, by Claim~$3$, there are three cases regarding the type of the neighborhood of a $2$-path $p = uvw$ from $\mathcal{D}$ depicted in Fig.~\ref{fig:edgetypes}.
	\begin{figure}[htp!]
		$$
			\includegraphics{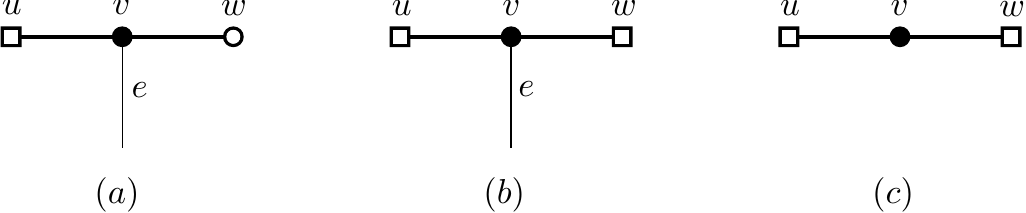}
		$$
		\caption{Three possible neighborhoods of a $2$-path in $\mathcal{D}$ with at least $4$ conflicts. 
			The empty square vertices denote pendant vertices for all incident elements of $\mathcal{D}$, 
			and the empty circle vertex denotes a central vertex for some element of $\mathcal{D}$.}
		\label{fig:edgetypes}
	\end{figure}	
	
	We first show that there are no edges of type $(a)$ in $G$.
	
	\medskip
	\noindent\textbf{Claim $4$.} \textit{No pendant vertex of any $2$-path from $\mathcal{D}$ is the central vertex of any element from $\mathcal{D}$.}
	
	\smallskip
	\begin{proofclaim}
		Let $p$ be a $2$-path from $\mathcal{D}$ with one pendant vertex being the central vertex of some $2$-path,
		and let the vertices of $p$ be denoted as in Fig.~\ref{fig:edgetypes}$(a)$.
		In this case, $p$ has exactly four conflicts, directly implying that $u$ and $w$ have degree $3$.
		By minimality, $G - p$ admits a locally irregular $4$-edge-coloring $\varphi$. 
		In the case when the elements in conflict with $p$ use at most three distinct colors in $\varphi$, 
		we complete the coloring using the fourth color on the edges of $p$. 
		So, we may assume that $p$ is incident with the edges of four distinct colors in $\varphi$. 
		Additionally, let $q$ and $r$ be the elements of $\mathcal{D}$ incident to $p$ at the vertex $w$ and $v$, respectively.
		Say that $\varphi(q) = 3$ and $\varphi(r) = 4$ (hence $\varphi(e) = 4$).

		Consider two subcases. In the former, suppose $r$ is a $2$-path. 
		Now, take a locally irregular edge-coloring of the graph $G - \set{p,r}$ with at most $4$ colors, 
		color $r$ with a free color (which is possible since $r$ has at most three conflicts in $G-p$), and finally
		color $p$ with either a free color or the same color as $r$. 
		Hence, in the latter case, two adjacent elements of $\mathcal{D}$ have the same color (see Fig.~\ref{fig:case_a}).
		\begin{figure}[htp!]
			$$
				\includegraphics{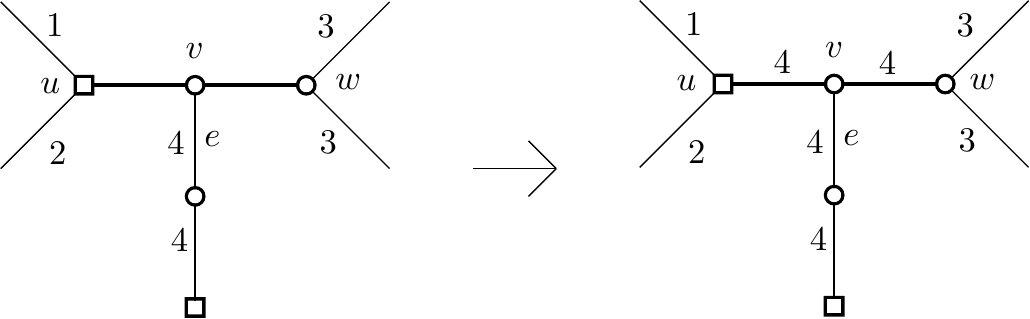}
			$$
			\caption{Coloring a path with the neighborhood depicted in Fig.~\ref{fig:edgetypes}$(a)$, when the path $r$ has three additional 
				conflicts, each colored with a distinct color different from $4$.}
			\label{fig:case_a}
		\end{figure}			

		In the second subcase, let $r$ be isomorphic to $K_{1,3}$ or $K_{1,3}''$. 
		Then, the element $q$ is a $2$-path, since it is colored differently as $r$ and it is hence not isomorphic to $K_{1,3}$ or $K_{1,3}''$.
		If $q$ is not incident to any edge of color $3$, then we color $p$ with $3$.
		Note that also in this case, two adjacent elements of $\mathcal{D}$ have the same color.

		Hence, we may assume that $q$ is incident with two edges (resp. four) of color $3$, both edges belonging to the same $2$-path (resp. all four edges belonging to two $2$-paths), 
		due to assumption $(ii)$.
		This means that $q$ is incident to at most three distinct colors, and hence we can recolor it with the color not appearing in its neighborhood.
		Finally, we color $p$ with $3$.
	\end{proofclaim}
	
	\medskip
	Thus, all the $2$-paths in $\mathcal{D}$ are of type $(b)$ or $(c)$. Moreover, the number of $2$-paths of type $(b)$ is limited.

	\medskip
	\noindent\textbf{Claim $5$.} \textit{There are at most two $2$-paths of type $(b)$ in $\mathcal{D}$. Moreover, any $2$-path of type $(b)$ is incident with a pendant vertex of $K_{1,3}$ 
		at the central vertex.}

	\smallskip
	\begin{proofclaim}
		By Claims~$1$ and~$4$, if the central vertex of a $2$-path from $\mathcal{D}$ is incident to some element $q$ from $\mathcal{D}$,
		then $q$ is isomorphic to $K_{1,3}$.
		Since $q$ has at least four conflicts by Claim~$2$, it follows that at most two of its pendant vertices are 
		the central vertices of some $2$-paths. This establishes the claim.
	\end{proofclaim}
	
	\medskip	
	We proceed by considering the edges of type $(b)$ in $G$. Before we state the claim and its proof, 
	we introduce some additional notation. By $G_{\varphi}^{\{a,b\}}(v)$, we denote the connected subgraph of $G$ 
	induced by the edges of colors $a$ and $b$ in the coloring $\varphi$ containing the vertex $v$. If $G_{\varphi}^{\{a,b\}}(u) = G_{\varphi}^{\{a,b\}}(v)$
	for some pair of vertices $u$ and $v$, we say that there is an \textit{$(a,b)$-path between $u$ and $v$}.
	Further, we say that we \textit{make a swap in $G_{\varphi}^{\{a,b\}}(v)$}, if we recolor all the edges of color $a$ in $G_{\varphi}^{\{a,b\}}(v)$ with $b$
	and vice versa. Clearly, if $\varphi$ is a locally irregular edge-coloring, then it remains such after an arbitrary number of swappings.
	
	\medskip
	\noindent\textbf{Claim $6$.} \textit{No pendant vertex of $K_{1,3}$ is the central vertex of any $2$-path from $\mathcal{D}$.}
	
	\smallskip
	\begin{proofclaim}
		Let $p$ be a $2$-path from $\mathcal{D}$ with the central vertex $v$ being a pendant vertex of $K_{1,3}$, which we call $r$.
		By Claim~$4$, $p$ is of type $(b)$.
		Label the vertices as in Fig.~\ref{fig:edgetypes}$(b)$, with $e$ being the edge of $K_{1,3}$.
		In this case, $p$ has either four or five conflicts. 
		In the former case, one of the vertices $u$ and $w$ is of degree $2$ and the other is of degree $3$; 
		in the latter case, both are of degree $3$.

		Let $x$ be the vertex of degree $3$ in $r$. 
		By minimality of $G$, $G - p$ admits a locally irregular $4$-edge-coloring $\varphi$, and assume $\varphi(r) = 4$. 
		As before, we may assume that $p$ is incident to all four colors, for otherwise we simply color the edges of $p$
		with the color $p$ is not incident with. 

		Now, we consider two cases regarding the number of conflicts of $p$.			
		Suppose first that $p$ has four conflicts and let $u$ be the vertex of degree $2$. 
		Let $q$ be the second $2$-path incident to $u$ and say $\varphi(q) = 1$. 
		By Claim~$2$, it must be of type $(b)$, and, by Claim~$5$, incident to $r$ at the central vertex.
		Hence, the third pendant vertex of $r$ is a pendant vertex of two $2$-paths, colored with $2$ and $3$, 
		otherwise we would recolor $r$ with a free color and color $p$ with $4$.		
		Let $s$ and $t$ be the two $2$-paths incident with $p$ at the vertex $w$.
		Without loss of generality, assume $\varphi(s) = 2$ and $\varphi(t) = 3$.
		By Claim~$5$, both, $s$ and $t$ are of type $(c)$, having three conflicts in $G - p$. 
		Furthermore, $s$ is incident to all three colors distinct from $2$, 
		and $t$ is incident to all three colors distinct from $3$.
		Now, if $G_{\varphi}^{\{2,4\}}(w) \neq G_{\varphi}^{\{2,4\}}(x)$, then we make a swap in $G_{\varphi}^{\{2,4\}}(w)$, 
		hence recoloring $s$ with $4$, and obtaining $2$ as a free color for $p$. 
		We proceed similarly in the case when $G_{\varphi}^{\{3,4\}}(w) \neq G_{\varphi}^{\{3,4\}}(x)$.
		So, we may assume $G_{\varphi}^{\{2,4\}}(w) = G_{\varphi}^{\{2,4\}}(x)$ 
		and $G_{\varphi}^{\{3,4\}}(w) = G_{\varphi}^{\{3,4\}}(x)$
		(see Fig.~\ref{fig:type_b}).
		Then, we recolor $s$ with $3$ and $t$ with $2$. 
		Observe that now $G_{\varphi}^{\{3,4\}}(w) \neq G_{\varphi}^{\{3,4\}}(x)$ (due to the fact that all $2$-paths except for $p$ and $q$ are of type $(c)$, 
		implying that $G_{\varphi}^{\{3,4\}}(w)$ is isomorphic to a path, possibly with a pendant $K_{1,3}$).
		So, we can make a swap on $G_{\varphi}^{\{3,4\}}(w)$ and finally color $p$ with $3$.		
		\begin{figure}[htp!]
			$$
				\includegraphics{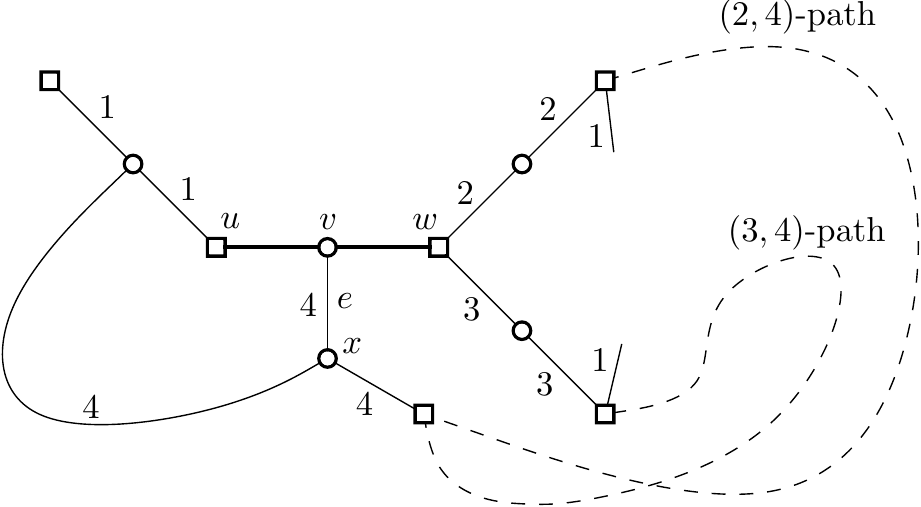}
			$$
			\caption{The case with $p$ having four conflicts and $G_{\varphi}^{\{2,4\}}(w) = G_{\varphi}^{\{2,4\}}(x)$,
				$G_{\varphi}^{\{3,4\}}(w) = G_{\varphi}^{\{3,4\}}(x)$.}
			\label{fig:type_b}
		\end{figure}

		So, we may assume that $p$ has five conflicts. 
		We again consider two subcases, now regarding the number of edges of color $4$ incident to $p$.
		Suppose first that $e$ is the only edge incident to $p$ colored with $4$.
		Let $q_1$, $q_2$ be the two $2$-paths incident to $p$ at $u$, and $s_1$, $s_2$ the two $2$-paths incident to $p$ at $w$.
		Without loss of generality, we may assume $\varphi(q_1) = \varphi(s_1) = 1$, $\varphi(q_2) = 2$, and $\varphi(s_2) = 3$.
		By Claim~$5$, at most one of these four $2$-paths is of type $(b)$, and in that case the central vertex of that path is 
		incident to $K_{1,3}$. So, we may assume that both paths $s_1$ and $s_2$ are of type $(c)$.
		The $2$-path $q_2$ (resp. $s_2$) is incident with the edges of colors $1$, $3$, and $4$ (resp. $1$, $2$, and $4$)
		otherwise we recolor it with a free color and color $p$ with $2$ (resp. $3$).				
		Next, similarly as above, if $G_{\varphi}^{\{2,3\}}(u) \neq G_{\varphi}^{\{2,3\}}(w)$, we make a swap on $G_{\varphi}^{\{2,3\}}(u)$.
		Hence, we may assume $G_{\varphi}^{\{2,3\}}(u) = G_{\varphi}^{\{2,3\}}(w)$. 
		If $s_2$ is the only element incident to $s_1$ colored with $3$, we recolor $s_1$ with $3$ and $s_2$ with $1$. 
		Then $G_{\varphi}^{\{2,3\}}(u) \neq G_{\varphi}^{\{2,3\}}(w)$, and we make a swap on $G_{\varphi}^{\{2,3\}}(w)$.
		If $s_1$ is not incident with an element of color $2$, we recolor it with $2$ and recolor $s_2$ with $1$. 
		Hence, $s_1$ is not incident with an element of color $4$. In this case, we recolor $s_1$ with $4$ and recolor $s_2$ with $1$. 
		In all the three cases listed above, the recoloring yields a free color for $p$ 
		with which we complete the coloring $\varphi$ on $G$.
		
		Therefore, we may assume that there is another edge, apart from $e$, incident to $p$ colored with $4$. 		
		Suppose first that one of the vertices $u$ and $w$, say $w$, is also incident with $r$. 
		Then $v$ is the only central vertex incident to $r$, 
		for otherwise $r$ would have at most tree conflicts in $G$.
		Clearly, there are three $2$-paths (of type $(c)$) incident to $p$, $q_1$ and $q_2$ at $u$, and $s_1$ at $w$, 
		colored, respectively, with the colors $1$, $2$, and $3$ in $\varphi$.
		As above, we may assume $q_1$, $q_2$, and $s_1$ are incident with the edges of all colors except for theirs.
		So, if $G_{\varphi}^{\{2,3\}}(u) \neq G_{\varphi}^{\{2,3\}}(w)$ or $G_{\varphi}^{\{1,3\}}(u) \neq G_{\varphi}^{\{1,3\}}(w)$,
		we make a corresponding swap obtaining a free color for $p$. 
		Hence, $G_{\varphi}^{\{2,3\}}(u) = G_{\varphi}^{\{2,3\}}(w)$ and $G_{\varphi}^{\{1,3\}}(u) = G_{\varphi}^{\{1,3\}}(w)$.
		Now, we recolor $q_1$ with $2$ and $q_2$ with $1$, and make a swap on $G_{\varphi}^{\{2,3\}}(w)$. This yields a free color $3$ for $p$.
		
		Thus, $r$ is not incident to $w$, and there are four $2$-paths incident to $p$ at $u$ and $w$ (as above denoted by $q_1$, $q_2$, $s_1$, and $s_2$), 
		colored with four distinct colors, say $\varphi(q_1) = 1$, $\varphi(q_2) = 2$, $\varphi(s_1) = 3$, and $\varphi(s_2) = 4$.
		Observe that all four $2$-paths are incident to the elements of all three colors distict from theirs, otherwise we can recolor them
		obtaining a free color for $p$, or establishing the previous case (if $s_2$ could be recolored).
		We use the same reasoning as above to establish that $G_{\varphi}^{\{2,3\}}(u) = G_{\varphi}^{\{2,3\}}(w)$ and $G_{\varphi}^{\{1,3\}}(u) = G_{\varphi}^{\{1,3\}}(w)$.
		Hence, we can recolor $q_1$ with $2$ and $q_2$ with $1$, and finally make a swap on $G_{\varphi}^{\{1,3\}}(u)$.
		This yields a free color $3$ for $p$ and establishes the claim.
	\end{proofclaim}
	
	\medskip
	From Claims~$1$, $4$, and $6$ it follows that all $2$-paths in $\mathcal{D}$ are of type $(c)$ and moreover, 
	by Claim~$2$, every pendant vertex of a $2$-path is of degree $3$. 
	
	\medskip
	\noindent\textbf{Claim $7$.} \textit{The graph $G$ has an odd number of edges.}
	
	\smallskip
	\begin{proofclaim}
		Suppose that $G$ has an even number of edges. 
		Then it is isomorphic to a cubic graph with every edge subdivided once. 
		By Vizing's theorem, every cubic graph admits a proper edge-coloring with at most $4$ colors, 
		which consequently induces a locally irregular $4$-edge-coloring of $G$.
	\end{proofclaim}
	
	\medskip
	Hence, $G$ has an odd number of edges, every $2$-path in $\mathcal{D}$ is of type $(c)$, and there
	is an element $r$ in $\mathcal{D}$ isomorphic to either $K_{1,3}$ or $K_{1,3}''$.
	Let $x$ be the vertex of degree $3$ of $r$, and let $v_1$, $v_2$, and $v_3$ be the pendant vertices of $r$.
	In the case when $r$ is isomorphic to $K_{1,3}''$, we may assume, without loss of generality, that $v_1$ is adjacent to $x$ and set $e_1 = v_1 x$.
	
	Let $G^*$ be the graph obtained from $G - e_1$ by contracting all vertices of degree $2$ except $v_1$, 
	i.e. we remove every vertex of degree $2$ and connect its two neighbors by an edge.
	Notice that $v_2$ and $v_3$ are adjacent in $G^*$ even in the case when $r$ is isomorphic to $K_{1,3}''$.
	By Vizing's theorem, $G^*$ admits an edge-coloring $\varphi^*$ with at most $4$ colors.
	The coloring $\varphi^*$ induces a locally irregular edge coloring $\varphi$ of $G - x$, if we assign 
	the color of an edge $e = uv$ in $G^*$ to the two edges in $G$ corresponding to $e$ before contracting the vertex
	of degree $2$ neighboring both, $u$ and $v$. 	
	The edges $xv_2$ and $xv_3$ receive the same color, say $1$, in $\varphi$, since $v_2v_3$ are adjacent in $G^*$,
	while $e_1$ remains non-colored. 
		
	In what follows, we show that we can modify $\varphi$ such that $e_1$ receives the same color as the remaining edges of $r$ 
	and hence $\varphi$ becomes a locally irregular edge-coloring of $G$.			
	We may assume that $v_1$ is incident to an edge of color $1$ (otherwise we simply color $e_1$ with $1$ and we are done). 
	Without loss of generality, we may further assume that $v_1$ is also incident with an edge of color $2$ (see Fig.~\ref{fig:case_c}).
	\begin{figure}[htp!]
		$$
			\includegraphics{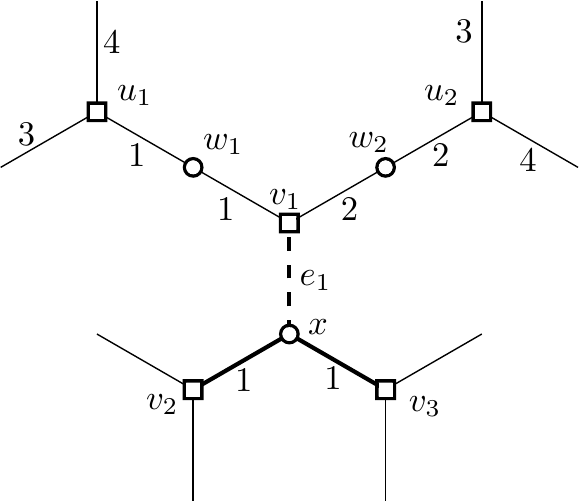}
		$$
		\caption{The initial situation after coloring the graph $G^*$ with $4$ colors inducing the coloring $\varphi$ in $G$.}
		\label{fig:case_c}
	\end{figure}		
	
	Denote the vertices of the two $2$-paths incident to $v_1$ as depicted in Fig.~\ref{fig:case_c}.
	We refer to the path $u_1w_1v_1$ as $q_1$, and to $u_2w_2v_1$ as $q_2$.
	We may assume that $u_1$ is incident with edges of colors $3$ and $4$, otherwise we recolor $q_1$ with $3$ or $4$, respectively. 
	Similarly, the vertex $u_2$ is incident with edges of colors $3$ and $4$, 
	otherwise we recolor $q_1$ with $2$ and $q_2$ with $3$ or $4$, which is not incident to $u_2$.
			
	Similarly as in the proof of Claim~$6$, we may assume that $G_{\varphi}^{\{1,3\}}(v_1) = G_{\varphi}^{\{1,3\}}(x)$ and
	$G_{\varphi}^{\{1,4\}}(v_1) = G_{\varphi}^{\{1,4\}}(x)$, for otherwise we make a swap on $G_{\varphi}^{\{1,3\}}(v_1)$ or $G_{\varphi}^{\{1,4\}}(v_1)$,
	and hence making $1$ a free color for $e_1$.
	Furthermore, since we can exchange the colors of $q_1$ and $q_2$, we also have
	$G_{\varphi}^{\{1,3\}}(u_2) = G_{\varphi}^{\{1,3\}}(x)$ and $G_{\varphi}^{\{1,4\}}(u_2) = G_{\varphi}^{\{1,4\}}(x)$ (see Fig.~\ref{fig:case_c2}).
	\begin{figure}[htp!]
		$$
			\includegraphics{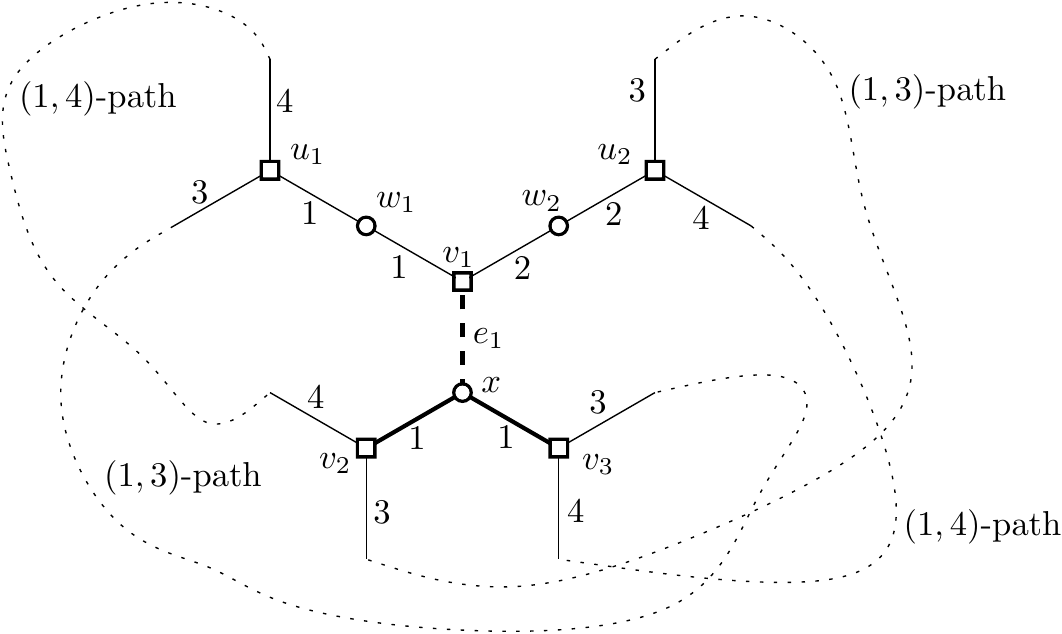}
		$$
		\caption{There are $(1,3)$-paths and $(1,4)$-paths between the vertex $x$ and the vertices $v_1$ and $u_2$.}
		\label{fig:case_c2}
	\end{figure}		
			
	Finally, we have $G_{\varphi}^{\{2,3\}}(v_1) = G_{\varphi}^{\{2,3\}}(u_1)$, 
	otherwise we make a swap on $G_{\varphi}^{\{2,3\}}(v_1)$, recolor $q_1$ with $2$ and color $e_1$ with $1$.
	Analogously, we have $G_{\varphi}^{\{2,4\}}(v_1) = G_{\varphi}^{\{2,4\}}(u_1)$.
	This means that the two $2$-paths incident to $u_1$ distinct from $q_1$ are incident only to edges of colors $1$ and $2$, 
	and hence we can exchange their colors. 
	Consequently, after such exchange $G_{\varphi}^{\{2,3\}}(v_1) \neq G_{\varphi}^{\{2,3\}}(u_1)$,
	so we can make a swap on $G_{\varphi}^{\{2,3\}}(v_1)$, recolor $q_1$ with $2$ and color $e_1$ with $1$.
	Thus, we obtained a locally irregular edge-coloring of $G$ with at most $4$ colors satisfying the properties $(i)$ and $(ii)$, 
	a contradiction which establishes the theorem.
\end{proof}

The proof of Theorem~\ref{thm:main} trivially follows from Theorems~\ref{thm:decompose} and~\ref{thm:mainext}.
Unfortunately, we were not able to confirm Conjecture~\ref{con:main} for subcubic graphs. Our approach, using decompositions into $2$-paths and (generalized) claws
such that each element of the decomposition is colored with one color does not work for three colors in general. 
In Fig.~\ref{fig:cube}, an example of a graph with a given strongly pertinent decomposition is presented. 
On the other hand, choosing another strongly pertinent decomposition allows such a coloring of the given graph (see the right graph in Fig.~\ref{fig:cube}).
Therefore, we believe that after some modification of our technique, one can be able to prove Conjecture~\ref{con:main} for subcubic graphs.
\begin{figure}[htp!]
	$$
		\includegraphics{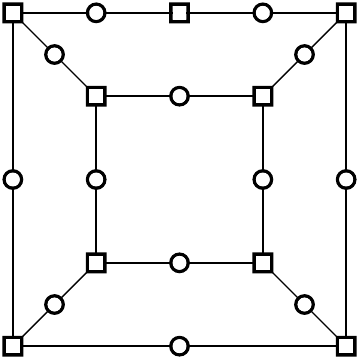} \quad\quad\quad\quad\quad
		\includegraphics{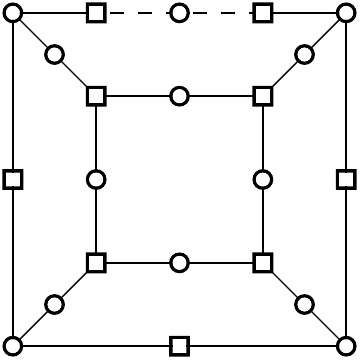}
	$$
	\caption{The left graph of $3$-dimensional cube with each edge subdivided once except for one which is subdivided three times, 
		with the given strongly pertinent decomposition does not admit a locally irregular edge-coloring
		with at most three colors such that every element of the decomposition receives only one color. The same graph on the right side 
		with another strongly pertinent decomposition admits such a coloring with two colors; all the elements are colored with the same color except the path depicted dashed.}
	\label{fig:cube}
\end{figure}

Finally, using the computer, we verified that every cubic graph on at most $20$ vertices, and every decomposable subcubic graph on at most $15$ vertices with minimum degree $2$ 
admits a locally irregular edge-coloring with at most $3$ colors.

\paragraph{Acknowledgment.} 

The first author was partly supported by the Slovenian Research Agency Program P1--0383 and by the National Scholarship Programme of the Slovak Republic.
The second author was supported by the National Science Centre, Poland, Grant No. 2014/13/B/ST1/01855 and partly supported by the Polish
Ministry of Science and Higher Education. The third author was supported by the Slovak Research and Development Agency under 
the Contract No. APVV--15--0116 and by the Slovak VEGA Grant 1/0368/16.

\bibliographystyle{plain}

\bibliography{mainBib}

\end{document}